\documentclass{scrartcl}
\usepackage{amsthm,amssymb,amsmath}
\usepackage{enumitem}
\usepackage{hyperref}
\usepackage{subcaption}
\usepackage{xspace}
\usepackage[utf8]{inputenc}
\usepackage[a4paper]{geometry}

\usepackage{pgf,tikz,ifthen,calc}
\usetikzlibrary{math,positioning,calc,arrows.meta,decorations.pathmorphing,decorations.markings,shapes,hobby}

\newtheorem{theorem}{Theorem}[section]
\newtheorem{lemma}[theorem]{Lemma}

\newtheorem{corollary}[theorem]{Corollary}

\theoremstyle{definition}

\theoremstyle{remark}

\newcommand{\N}{\ensuremath{\mathbb N}}
\newcommand{\Sym}{\ensuremath{\mathcal S}}
\newcommand{\Prob}{\ensuremath{\mathbb P}}
\newcommand{\Gcal}{\ensuremath{\mathcal G}}
\newcommand{\Hcal}{\ensuremath{\mathcal H}}
\newcommand{\Bcal}{\ensuremath{\mathcal B}}

\title{A note on classes of subgraphs of locally finite graphs}
\author{Florian Lehner}

\begin{document}

\maketitle
\begin{abstract}
We investigate the question how `small' a graph can be, if it contains all members of a given class of locally finite graphs as subgraphs or induced subgraphs.
More precisely, we give necessary and sufficient conditions for the existence of a connected, locally finite graph $H$ containing all elements of a graph class $\Gcal$. 
These conditions imply that such a graph $H$ exists for the class $\Gcal_d$ consisting of all graphs with maximum degree $<d$ which raises the question whether in this case $H$ can be chosen to have bounded maximum degree. We show that this is not the case, thereby answering a question recently posed by Huynh et al.
\end{abstract}

\section{Introduction}

Given a graph class $\Gcal$, we call a graph $G \in \Gcal$ \emph{(strongly) universal} for $\Gcal$, if it contains every graph in $\Gcal$ as an (induced) subgraph. This concept was probably first studied by Rado who showed in \cite{zbMATH03225943} that the set of all countable graphs contains a strongly universal element. 

Numerous other results about the existence or non-existence of universal elements in various graph classes can be found in the literature. Graph classes defined by excluded subgraphs are particularly well studied, especially when the excluded subgraphs are trees or cycles, see for instance \cite{zbMATH05143571,zbMATH06806086,zbMATH00870145,zbMATH01011260,zbMATH01321812,zbMATH04099373,zbMATH03877235}.

In this short note we are interested in the question whether there is a locally finite (that is, all vertices have finite degree) graph containing all members of a given graph class as (induced) subgraphs. Let us say that a graph $H$ contains a graph class $\Gcal$, if $H$ contains every $G \in \Gcal$ as a subgraph, and that $H$ strongly contains $\Gcal$, if $H$ contains every $G \in \Gcal$ as an induced subgraph. 

Our first main result states that the existence of a locally finite graph $H$ which (strongly) contains $\Gcal$ can be determined by looking at the set of balls of finite radius appearing in members of $\Gcal$. 

To state the theorem, we need to set up some notation.
Given a graph class $\Gcal$, let 
\[
\Bcal = \{B_G(v,r) \mid G \in \Gcal, v \in V(G), r \in \N_0\}
\] where $B_G(v,r)$ and $B_{G'}(v',r)$ are considered the same if there is an isomorphism between them which maps $v$ to $v'$. Define a tree structure $T(\Gcal)$ on $\Bcal$ by connecting $B_G(v,r-1)$ and $B_G(v,r)$ by an edge for every $G \in \Gcal$, $v \in V(G)$, and $r \in \N$. 

Let us call $\Gcal$ closed, if for every $G \notin \Gcal$ there is some $v \in V(G)$ and $r \in \mathbb N$ such that $B_G(v,r) \notin \Bcal$. 
We note that many interesting graph classes are closed. For instance it is easy to see that any graph class defined by forbidden (induced) subgraphs of finite diameter is closed, and the same is true for graph classes defined by finite forbidden minors or topological minors.

With the above notation, and denoting by $T_\infty$ the tree in which every vertex has countably infinite degree, we have the following characterisation of closed graph classes which admit a locally finite graph containing them.


\begin{theorem}
\label{thm:universal}
Let $\Gcal$ be a closed class of connected, locally finite graphs. The following three statements are equivalent.
\begin{enumerate}
    \item There is a connected, locally finite graph which strongly contains $\Gcal$.
    \item There is a connected, locally finite graph which contains $\Gcal$.
    \item $T(\Gcal)$ does not contain a subdivision of $T_\infty$.
\end{enumerate}
\end{theorem}

Closedness of graph classes has a topological interpretation which adds another interesting facet to the above theorem: the three equivalent conditions are met if and only if the set
$
\Gcal^\bullet= \{(G,v) \mid G \in \Gcal, v \in V(G)\},
$
is $\sigma$-compact with respect to a natural topology on the set of rooted graphs.

We also note that a similar result holds, if we drop the requirement that $\Gcal$ is closed (see Theorem \ref{thm:universal-notclosed}), but the third statement has to be replaced by a more technical condition.

The second main result of this paper concerns graphs not containing some fixed star.
Denote by $\Gcal_d$ the class of all connected graphs which do not contain a star with $d$ leaves as a subgraph. Equivalently, $\Gcal_d$ is the class of all graphs in which all vertex degrees are strictly less than $d$. We also allow $d = \infty$ and let $\Gcal_\infty$ denote the class of connected, locally finite graphs.

The class $\Gcal_\infty$ does not contain a universal element by an argument attributed to de Bruijn in \cite{zbMATH03225943} (or by Theorem \ref{thm:universal}), and the same is true for $\Gcal_d, d \geq 4$ as shown in \cite{zbMATH03877235}. However, it follows from Theorem \ref{thm:universal} that there is a graph $H \in \Gcal_\infty$ (strongly) containing $\Gcal_d$ for $d < \infty$. We show that this is best possible, thereby answering a question recently posed by Huynh et al. in \cite[Section 7]{huynh2021universality}.

\begin{theorem}
\label{thm:unbounded}
If a countable graph $H$ contains $\Gcal_4$, then it necessarily has unbounded vertex degrees.
\end{theorem}

Besides the result itself, the proof method may be of interest. Many results about non-existence of certain countable universal graphs are based on the fact that partitioning an uncountable set into countably many parts yields at least one uncountable part. The key idea in our proof can be seen as a quantitative version of this infinite pigeonhole principle: if the uncountable set is endowed with a non-null measure and the parts are measurable, then at least one of the parts must have positive measure.

\section{Preliminaries}

Throughout this short note, all graphs are assumed to be connected and simple unless explicitly stated otherwise. Moreover we assume all graphs to be countable in order to avoid set-theoretical subtleties which could arise due to the fact that the class of all graphs is not a set. As usual, denote by $V(G)$ and $E(G)$ the vertex and edge set of a graph $G$ respectively.

An \emph{embedding} of a graph $G$ into a graph $H$ is a map $\iota\colon V(G) \to V(H)$ which preserves adjacency, that is, $uv \in E(G) \implies \iota(u)\iota(v) \in E(H)$. A \emph{strong embedding} is an embedding $\iota$ which additionally preserves non-adjacency, that is, $uv \in E(G) \iff \iota(u)\iota(v) \in E(H)$. Note that $G$ is a subgraph of $H$ if and only if there is an embedding of $G$ into $H$, and $G$ is an induced subgraph of $H$ if and only if there is a strong embedding of $G$ into $H$.

A \emph{graph class} $\Gcal$ is a set of graphs such that  $G \in \Gcal$ whenever $G$ is isomorphic to some $G' \in \Gcal$. We say that a graph $H$ \emph{(strongly) contains} a graph class $\Gcal$, if every $G \in \Gcal$ has a (strong) embedding into $H$. A graph $G\in \Gcal$ which (strongly) contains $\Gcal$ is called \emph{(strongly) universal} for $\Gcal$.

A \emph{rooted graph} is a pair $(G,v)$ where $G$ is a graph and $v \in V(G)$. All of the above definitions carry over to rooted graphs, the only difference is that an embedding of a rooted graph $(G,v)$ into a rooted graph $(H,w)$ must always map the root $v$ of $G$ to the root $w$ of $H$. 
For a graph class $\Gcal$, we define the associated rooted graph class by 
\[
\Gcal^\bullet = \{(G,v) \mid G \in \Gcal, v \in V(G)\},
\]
in other words, $\Gcal^\bullet$ is the class of all rooted graphs whose underlying graph lies in $\Gcal$.

Any class of rooted graphs can be endowed with a metric by letting 
\[
d((G,v),(G',v')) = \exp(-\max \{r \mid B_G(v,r) \text{ is isomorphic to } B_{G'}(v',r)\}),
\]
where $B_G(v,r)$ as usual denotes the ball in $G$ with radius $r$ and centre $v$. 

A class $\Gcal$ of locally finite graphs is \emph{closed} if $\Gcal^\bullet$ is a closed subset of the class $\Gcal_\infty^\bullet$ of all locally finite, rooted graphs with respect to the topology induced by the metric given above. We remark that this definition of closedness is equivalent to the one given in the introduction.

We slightly abuse notation, and extend the definition of the tree $T(\Gcal)$ given in the introduction to rooted graph classes as follows. The vertex set of $T(\Gcal)$ is $\Bcal(\Gcal) := \{B_G(v,r) \mid (G,v) \in \Gcal, r \in \N_0\}$ where $B_G(v,r)$ and $B_{G'}(v',r)$ are considered the same if there is an isomorphism between them which maps $v$ to $v'$. Note that isomorphic balls with different radii are considered different, for instance given a finite rooted graph $(G,v)$ the balls $B_G(v,r)$ are distinct elements of $\Bcal(\Gcal)$ although they are isomorphic for all but finitely many $r$. Connect $B_G(v,r-1)$ and $B_G(v,r)$ by an edge for every $(G,v) \in \Gcal$ and $r \in \N$. For a class $\Gcal$ of (unrooted) graphs, the definition of $T(\Gcal)$ from the introduction coincides with the tree $T(\Gcal^{\bullet})$ corresponding to the associated rooted graph class.

We treat $T(\Gcal)$ as a rooted tree with root $B_0 = B_G(v,0)$; note that the definition of $B_0$ does not depend on the specific choice of $(G,v) \in \Gcal$ because $B_G(v,0)$ is the trivial one vertex graph for every choice of $(G,v)$. We call $B \in \Bcal(\Gcal)$ an \emph{ancestor} of $B' \in \Bcal(\Gcal)$ and $B'$ a \emph{descendant} of $B$, if $B$ lies on the $B_0$--$B'$-path in $T(\Gcal)$. Note that in this case there is $(G,v) \in \Gcal$ and $i\leq j$ such that $B$ is isomorphic to $B_G(v,i)$ and $B'$ is isomorphic to $B_G(v,j)$.

The tree $T(\Gcal)$ is closely linked to the topology of $\Gcal$ induced by the metric defined above. Note that any element $(G,v)$ of a rooted graph class $\Gcal$ corresponds to a one-way infinite path $(B_0, B_1, B_2, B_3, \dots)$ in $T(\Gcal)$ where each $B_i$ is isomorphic to $B_G(v,i)$; if we didn't treat isomorphic balls of different radii as different objects, paths corresponding to graphs of finite diameter would be finite. Mapping each graph to the corresponding path gives a homeomorphism between $\Gcal$ and a subset of the end space (see \cite{zbMATH05936872} for an introduction) of $T(\Gcal)$ and the results given in the remainder of this section can be seen as straightforward consequences of this homeomorphism. We provide proofs of these results for the convenience of the reader.

The first result concerns the closure of a graph class $\Gcal$. Note that from an infinite path $p = (B_0, B_1, B_2, B_3, \dots)$ starting at the root of $T(\Gcal)$  we can construct a rooted graph $(G_p,v_p)$ as follows: Find $(G_i,v_i) \in \Gcal$ with $B_i = B_{G_i}(v_i,i)$, take a disjoint union of all $B_{G_i}(v_i,i)$ and identify $B_{G_{i-1}}(v_{i-1},i-1)$ and $B_{G_i}(v_i,i-1)$ via an automorphism---such an automorphism must exist, otherwise there would not be an edge connecting $B_{i-1}$ to $B_i$. Let $v_p$ be the vertex obtained from identifying all centres $v_i$. 

\begin{lemma}
\label{lem:inclosure}
Let $\Gcal$ be a class of countable rooted graphs. A rooted graph $(G,v)$ is in the closure $\overline \Gcal$ of $\Gcal$ if and only if it is of the form $(G_p,v_p)$ for some infinite path $p$ in $T(\Gcal)$ as above.
\end{lemma}
\begin{proof}
If $p$ is an infinite path as above, then $(G_p,v_p)$ satisfies $B_{G_p}(v_p,i) = B_i$. Since there is a graph $(G_i,v_i) \in \Gcal$ such that $B_i = B_{G_i}(v_i,i)$ and hence $d((G_p,v_p),(G_i,v_i)) \leq exp(-i)$ for every $i$ we conclude that $(G_p,v_p) \in \overline \Gcal$.

Conversely, if $(G,v)$ is contained in the closure, then for any $i \in \mathbb N$ there must be a graph $(G_i,v_i) \in \Gcal$ such that $d((G,v),(G_i,v_i)) \leq exp(-i)$, and the sequence of $B_{G_i}(v_i,i)$ gives the desired infinite path in $T(\Gcal)$.
\end{proof}

The second result provides a characterisation of compact and $\sigma$-compact graph classes; recall that a topological space is called $\sigma$-compact if it can be covered by countably many compact subsets, and let $T_\infty$ denote the tree in which every vertex has countably infinite degree.

\begin{lemma}
\label{lem:compactnesscriterion}
Let $\Gcal \subseteq \Gcal_\infty$ be a closed class of rooted graphs.
\begin{enumerate}
    \item $\Gcal$ is compact if and only if $T(\Gcal)$ is locally finite
    \item $\Gcal$ is $\sigma$-compact if and only if $T(\Gcal)$ does not contain a subdivision of $T_\infty$.
\end{enumerate}
\end{lemma}

\begin{proof}
For the first statement, note that if $T(\Gcal)$ is not locally finite, then there is some $i \in \N$ such that there are infinitely many non-isomorphic $B_{G}(v,i)$ with $(G,v) \in \Gcal$. Hence there is an infinite cover of $\Gcal$ whose elements are pairwise disjoint open balls of radius $\exp(-i+1)$. This cover clearly has no finite sub-cover, hence $\Gcal$ is not compact.

For the converse implication, recall that a metric space is compact if and only if it is sequentially compact. Let $(G_i,v_i)_{i \in \N}$ be a sequence of graphs in $\Gcal$. Since $T(\Gcal)$ is locally finite there are only finitely many non-isomorphic $k$-balls in $\Bcal(\Gcal)$ for every $k$. Hence, we can inductively find subsequences $(G_i,v_i)_{i \in I_k}$ with $\N \supseteq I_1 \supseteq I_2 \supseteq \dots$ such that $B_{G_i}(v_i,k)$ and $B_{G_j}(v_j,k)$ are isomorphic for any pair $i,j \in I_k$. Pick an increasing function $f(k)\colon \N \to \N$ such that $f(k) \in I_k$ for every $k$. The sequence $B_{G_{f(k)}}(v_{f(k)},k)$ forms an infinite path $p$ in $T(\Gcal)$. Since $\Gcal$ is closed, the rooted graph $(G_p,v_p)$ defined by this path lies in $\Gcal$ by Lemma~\ref{lem:inclosure}. Moreover $d((G_{f(k)},v_{f(k)}),(G_p,v_p)) \leq \exp(-k)$, so $(G_{f(k)},v_{f(k)})_{k \in \N}$ is a convergent subsequence of $(G_i,v_i)_{i \in \N}$.

Now let us turn to the second statement of the lemma. Let $\Gcal$ be $\sigma$-compact and pick a cover $\Gcal = \bigcup_{i \in \N} \Hcal_i$ where each $\Hcal_i$ is compact. Note that $T(\Hcal_i)$ is the subtree of $T(\Gcal)$ induced by $\Bcal(\Hcal_i)$, and that that this subtree is locally finite because $\Hcal_i$ is compact.

Assume for a contradiction that $T(\Gcal)$ contains a subdivision of $T_\infty$. Call a vertex $B$ a \emph{branch point}, if it corresponds to a vertex of the subdivision of $T_\infty$. Note that from every branch point $B$ we can find infinitely many edge disjoint paths to other branch points, and that all but one of these branch points are descendants of $B$. Since $T(\Hcal_i)$ is locally finite, the first edge of at least one of these paths does not lie in $T(\Hcal_i)$. Connectedness of $T(\Hcal_i)$ implies that the other endpoint of this path is a descendant of $B$ which does not lie in $\Bcal(\Hcal_i)$. 

Let $B_1$ be an arbitrary branch point, and for every $i \in \N$ let $B_{i+1}$ be a branch point which is a descendant of $B_i$ and does not lie in $\Bcal(\Hcal_i)$. Let $p$ be an infinite path starting at $B_0$ and passing through every $B_i$. The graph $(G_p,v_p)$ is contained in $\Gcal$ by Lemma~\ref{lem:inclosure}, but it is not contained in any $\Hcal_i$ since there is some $r_i$ such that $B_{G_p}(v_p,r_i)$ is isomorphic to $B_i \notin \Bcal(\Hcal_i)$.

Now assume that $T(\Gcal)$ does not contain a subdivision of $T_\infty$. We construct a decomposition of $\Gcal$ into countably many compact sets by transfinite induction. 

In step $\alpha$ of the construction, assume that we have defined compact sets $\Hcal_\beta \subseteq \Gcal$ for every $\beta<\alpha$ (for $\alpha=0$ this is vacuously true) and let $T_\alpha =  T(\Gcal\setminus \bigcup_{\beta < \alpha}\Hcal_\beta)$. If $T_\alpha$ has no vertex $B$ such that the subtree induced by all descendants of $B$ is locally finite, then $T_\alpha$ and thus also $T(\Gcal)$ contains a subdivision of $T_\infty$. Hence we can pick such a vertex $B_\alpha$ and let $\Hcal_\alpha$ be the set consisting of all $(G,v) \in \Gcal$ which correspond to an infinite path in $T_\alpha$ which passes through $B_\alpha$. The tree $T(\Hcal_\alpha)$ is the union of all these infinite paths in $T_\alpha$, it is locally finite because of our choice of $B_\alpha$, and hence $\Gcal_\alpha$ is compact. 

Note that no graph in $(\Gcal\setminus \bigcup_{\beta < \alpha}\Hcal_\beta) \setminus \Hcal_\alpha = \Gcal \setminus \bigcup_{\beta < \alpha+1}\Hcal_\beta$ contains a ball isomorphic to $B_\alpha$, hence $B_\alpha$ is not a vertex of any $T_\gamma$ for any $\gamma > \alpha$. The class $\Gcal$ only contains locally finite graphs, so all graphs in $\Bcal(\Gcal)$ are finite and thus $T(\Gcal)$ has only countably many vertices. Hence the above procedure will terminate after some countable number of steps thereby yielding the desired composition of $\Gcal$.
\end{proof}

\section{Graph classes with locally finite universal graphs}

In this section we investigate conditions for the existence of a locally finite graph containing a given graph class and prove Theorem \ref{thm:universal}. We start by giving a more technical condition which is true even if the graph class in question is not closed.

Recall that $\Gcal_\infty$ denotes the class of connected, locally finite graphs, and consequently $\Gcal_\infty^\bullet$ denotes the class of rooted connected, locally finite graphs.

\begin{theorem}
\label{thm:universal-notclosed}
Let $\Gcal \subseteq \Gcal_\infty$. The following are equivalent.
\begin{enumerate}
    \item 
    \label{itm:strongcontain-notclosed}
    There is $H \in \Gcal_\infty$ which strongly contains $\Gcal$.
    \item 
    \label{itm:contain-notclosed}
    There is $H \in \Gcal_\infty$ which contains $\Gcal$.
    \item 
    \label{itm:tree-notclosed}We can decompose $\Gcal^{\bullet}$ as a countable union $\bigcup_{i \in \N}\Hcal_i$ such that $T(\Hcal_i)$ is locally finite for every $i \in \N$.
\end{enumerate}
\end{theorem}

Note that the classes $\Hcal_i$ in statement \ref{itm:tree-notclosed} are classes of rooted graphs. Indeed, the theorem becomes false if we decompose $\Gcal$ instead of $\Gcal^\bullet$: even a class $\Gcal$ consisting of a single locally finite graph can contain infinitely many non-isomorphic balls of radius $1$.

\begin{proof}
Condition \ref{itm:strongcontain-notclosed} trivially implies condition \ref{itm:contain-notclosed}.

For the implication \ref{itm:contain-notclosed} $\implies$ \ref{itm:tree-notclosed} let $H \in \Gcal_\infty$ be a graph which contains $\Gcal$ and let $(v_i)_{i \in \N}$ be an enumeration of the vertices of $H$. For each vertex $v_i$ let $\Hcal_i'$  be the class of all rooted graphs $(G,v_i)$ where $G$ is a subgraph of $H$ that contains $v_i$. Since $H$ is locally finite there are only finitely many subgraphs of diameter at most $r$ containing any fixed $v_i$, and thus $T(\Hcal_i')$ is locally finite for every $i$. Since $H$ contains $\Gcal$ we know that $\Gcal^\bullet \subseteq \bigcup_{i \in \N}\Hcal_i'$. Letting $\Hcal_i = \Gcal^\bullet \cap \Hcal_i'$ gives the desired decomposition; $T(\Hcal_i)$ is locally finite because it is a subtree of  $T(\Hcal_i')$.

For the final implication \ref{itm:tree-notclosed} $\implies$ \ref{itm:strongcontain-notclosed}, we first note that it suffices to find for each $i$ a rooted graph $(H_i,v_i) \in \Gcal_\infty^\bullet$ which strongly contains $\Hcal_i$. From the disjoint union of these graphs, we can construct a connected, locally finite graph $H$ by connecting $v_i$ to $v_{i+1}$ every $i \in \N$. This graph $H$ strongly contains $\Gcal$ because any strong embedding of a graph $(G,v) \in \Hcal_i$ into $(H_i,v_i)$ induces a strong embedding of $H$ into $G$.

For the construction of the graphs $(H_i,v_i)$ we denote by $\Bcal_i$ the set $\{B_G(v,r) \mid (G,v) \in \Hcal_i, r \in \N\}$. As before, two balls $B_G(v,r)$ and $B_{G'}(v',r)$ are considered the same if there is an isomorphism between them which maps $v$ to $v'$. For $B = B_G(v,r) \in \Bcal$ we call $B_G(v,r-1)$ the \emph{interior} $B^\circ$ of $B$; note that this does not depend on the specific choice of $(G,v) \in \Hcal_i$. For each $B \in \Bcal$ we fix a strong embedding $\iota_B\colon B^\circ\to B$ which maps the centre of $B^\circ$ to the centre of $B$. Clearly such an embedding always exists; for instance we can pick some $(G,v)$ such that $B= B_G(v,r)$ and take the restriction of the identity map to $B_G(v,r-1)$. 

Now define a graph $H_i$ as follows. Start with a disjoint union of all $B \in \Bcal_i$, and for each $B \in \Bcal_i$ identify every $v \in V(B^\circ)$ with $\iota_B(v)$. Clearly, the central vertices of all balls in $\Bcal_i$ are identified to a single vertex of $H_i$, this vertex $v_i$ is the designated root of $(H_i,v_i)$.

The rooted graph $(H_i,v_i)$ is connected because each $B = B_G(v,r)$ is connected and all copies of root vertices of balls are identified. It follows from the definition of the maps $\iota_B$ that any pair of vertices which is identified has the same distance from the centres of the respective balls.
Consequently the vertex set of $B_{H_i}(v,r)$ consists of the vertices of all $B_G(v,r)$ for $(G,v) \in \Hcal_i$. Each such $B_G(v,r)$ is finite because $\Gcal \subseteq \Gcal_\infty$, and there are only finitely many non-isomorphic balls of radius $r$ since otherwise $T(\Hcal_i)$ would not be locally finite. Thus $H_i$ is locally finite.

It only remains to show that any $(G,v) \in \Hcal_i$ has a strong embedding into $(H_i,v_i)$. For this purpose, first note that if a vertex $u \in B$ is identified with a vertex $u' \in B'$ a vertex $v \in B$ is identified with a vertex $v' \in B'$, then $uv \in E(B) \iff u'v' \in E(B')$. In particular, since $B_G(v,r)$ is isomorphic to some $B \in \Bcal_i$, there is a strong embedding $\iota_r\colon B_G(v,r) \to (H_i,v_i)$ for every $r$. 

We can construct the desired strong embedding $\iota\colon (G,v) \to (H_i,v_i)$ from these strong embeddings $\iota_r$ by a standard compactness argument. Since $B_{H_i}(v,r)$ is finite for every $r$, there are only finitely many ways to embed $B_G(v,r)$ into $H_i$. Hence there is an infinite subset $I_1 \subseteq \N$ such that the restrictions of $\iota_i$ and $\iota_j$ to $B_G(v,1)$ coincide for any pair $i,j \in I_1$. Inductive application of this argument gives a set $I_r \subseteq I_{r-1}$ for every $r \geq 2$ such that the restrictions of $\iota_i$ and $\iota_j$ to $B_G(v,r)$ coincide for any pair $i,j \in I_r$.

For $x \in V(G)$ whose distance from $v$ is at most $r$, we pick any $i \in I_r$ define $\iota(x) = \iota_i(x)$. We note that this does not depend on the specific choice of $r$ or $i$ because $\iota_i(x)=\iota_j(x)$ for any $x \in B_G(v,r)$ and any pair $i,j \in I_r = \bigcup _{s \geq r} I_s$. Since all $\iota_i$ are strong embeddings, so is $\iota$.
\end{proof}

\begin{corollary}
\label{cor:universal}
For every $d \in \N$ there is a graph $H \in \Gcal_\infty$ strongly contains $\Gcal_d$.
\end{corollary}

\begin{proof}
There are only finitely many non-isomorphic graphs of diameter at most $r$ in which every vertex has degree less than $d$. Hence the tree $T(\Gcal_d) = T(\Gcal_d^\bullet)$ is locally finite and Theorem \ref{thm:universal-notclosed} concludes the proof.
\end{proof}

\begin{corollary}
\label{cor:universal-all}
There is a graph $H \in \Gcal_\infty$ which strongly contains $\Gcal = \bigcup_{d \in N}\Gcal_d$.
\end{corollary}

\begin{proof}
Decompose $\Gcal^\bullet = \bigcup_{d \in \N}\Gcal_d^\bullet$, note that  $T(\Gcal_d^\bullet)$ is locally finite, and apply Theorem~\ref{thm:universal-notclosed}.
\end{proof}

Theorem \ref{thm:universal-notclosed} can also be used to show that graph classes whose members have bounded growth admit a locally finite graph containing them. For example, recall that a graph $G$ is said to have polynomial growth if there is some polynomial $P$ and a vertex $v \in v(G)$ such that $B_G(v,r)$ contains at most $P(r)$ vertices for every $r \in \N$.

\begin{corollary}
\label{cor:universal-poly}
There is a graph $H \in \Gcal_\infty$ which strongly contains the class $\Gcal$ of all graphs of polynomial growth.
\end{corollary}

\begin{proof}
Denote by $\Gcal(a,b)$ the class of all rooted graphs $(G,v)$ such that $B_G(v,r)$ contains at most $a\cdot r^b$ vertices for every $r$. It is not hard to see that $\Gcal^\bullet = \bigcup_{a,b \in \N} \Gcal(a,b)$. There are only finitely many non-isomorphic connected graphs on at most $a\cdot r^b$ vertices, hence $T(\Gcal(a,b))$ is locally finite and we can apply Theorem~\ref{thm:universal-notclosed}.
\end{proof}

Polynomial growth in the above corollary can of course be replaced by any growth bound, as long as there is a countable set $(f_i)_{i \in N}$ of functions such that for every allowed growth function $g$ there is an $i$ such that $g(x) \leq f_i(x)$ for every $x \in \N$.

While the above corollaries demonstrate that Theorem \ref{thm:universal-notclosed} is useful for showing that there is a graph $H \in \Gcal_\infty$ containing a certain graph class, the theorem is not quite as useful for showing the non-existence of such a $H$. For closed graphs the following theorem (which implies Theorem \ref{thm:universal}) gives a condition which is easier to falsify.



\begin{theorem}
\label{thm:universal-closed}
Let $\Gcal \subseteq \Gcal_\infty$ be a closed graph class. The following are equivalent.
\begin{enumerate}
    \item 
    \label{itm:contain-closed}
    There is $H \in \Gcal_{\infty}$ which contains $\Gcal$.
    \item 
    \label{itm:sigmacompact-closed}
    $\Gcal^\bullet$ is $\sigma$-compact.
    \item 
    \label{itm:tree-closed}
    $T(\Gcal)$ contains no subdivision of $T_\infty$.
\end{enumerate}
\end{theorem}

\begin{proof}
Statements \ref{itm:sigmacompact-closed} and \ref{itm:tree-closed} are equivalent by Lemma~\ref{lem:compactnesscriterion}, hence it suffices to show the equivalence of \ref{itm:contain-closed} and \ref{itm:sigmacompact-closed}.

For the implication \ref{itm:contain-closed} $\implies$ \ref{itm:sigmacompact-closed} we first note that by Theorem \ref{thm:universal-notclosed} (\ref{itm:tree-notclosed}) we can decompose $\Gcal^\bullet$ into countably many sets $\Hcal_i$ such that $T(\Hcal_i)$ is locally finite. Since $\Gcal^\bullet$ is closed, this implies that $\Gcal^\bullet = \bigcup_{i \in \N} \overline \Hcal_i$, where $\overline \Hcal_i$ denotes the closure of $\Hcal_i$ in $\Hcal_\infty^\bullet$. Lemma \ref{lem:compactnesscriterion} implies that $\Gcal^\bullet$ is $\sigma$-compact as claimed.

For the converse implication \ref{itm:sigmacompact-closed} $\implies$ \ref{itm:contain-closed} let $\Gcal^\bullet = \bigcup_{i \in \N}\Hcal_i$ where every $\Hcal_i$ is compact. By Lemma \ref{lem:compactnesscriterion}, $T(\Hcal_i)$ is locally finite for every $i$ and Theorem \ref{thm:universal-notclosed} (\ref{itm:tree-notclosed}) finishes the proof.
\end{proof}

As a corollary to the above theorem we obtain a result due to de Brujin mentioned in the introduction.

\begin{corollary}
$\Gcal_\infty$ contains no universal element.
\end{corollary}
\begin{proof}
Note that $\Gcal_\infty$ is closed. Every ball of radius $n$ in a locally finite graph can be extended in infinitely many ways to a ball of radius $n+1$ in a locally finite graph. Hence $T(\Gcal_\infty) = T_\infty$ and by Theorem \ref{thm:universal-closed} there is no $H \in \Gcal_\infty$ containing $\Gcal_\infty$.
\end{proof}

Similar arguments can be made for many other graph classes. For instance, the exact same proof also shows the following result. 

\begin{corollary}
There is no graph $H \in \Gcal_\infty$ containing the class of all locally finite, connected, planar graphs.
\end{corollary}


\section{Graphs with bounded degrees}

By Corollary \ref{cor:universal}, there is a graph $H \in \Gcal_\infty$ containing $\Gcal_d$. On the other hand, it is known (see for instance \cite{zbMATH03877235}) that there is no $ H \in \Gcal_d$ containing $\Gcal_d$. This raises the natural question whether there is any $D < \infty$ such that there is a graph $H \in \Gcal_D$ containing $\Gcal_d$. Theorem \ref{thm:unbounded} from the introduction states that such a $D$ does not even exist for $\Gcal_4$; the remainder of this section is dedicated to the proof of this theorem.

\begin{proof}[Proof of Theorem \ref{thm:unbounded}]
Assume that there is a countable graph $H$ with maximum degree $\Delta$ containing $\Gcal_4$.

We define a subfamily of $\Gcal_4$ as follows. Let $[n]= \{1,2,\dots, n\}$ and $[m,n] = \{m,m+1,\dots, n\}$. Let $\Sym$ be the set of all bijective functions $s \colon \N \to \N$ such that $s([2^i]) = [2^i]$ for all $i \in \N$. For each $s \in \Sym$ let $G_s$ be the graph with vertex set $\mathbb Z$ and edges connecting $n$ to $n+1$ for every $n \in \mathbb Z$ as well as edges connecting $2n$ to $-2s(n)$ for every $n \in \mathbb N$; the reason we only connect even vertices is to get rid of unwanted automorphisms.
Denote by $G_s(i)$ the subgraph of $G_s$ induced by $[-2^{i+1},2^{i+1}]$. For every vertex $v \in V(H)$, let $\Sym(v,i)$ be the set of all $s \in \Sym$ such that there is an embedding of $G_s(i)$ into $H$ mapping $0$ to $v$.

Let $n(v,i)$ the number of non-isomorphic $G_s(i)$ with $s \in \Sym(v,i)$. 
Note that for any embedding of $G_s(i)$ which maps $0$ to $v$ there are at most $\Delta^{2^{i+1}}$ possibilities to embed the path induced by $[0,2^{i+1}]$, and similarly for the path induced by $[-2^{i+1},0]$. Once the paths are embedded, the image of $2n$ has fewer than $\Delta$ possible neighbours in the embedding of the path induced by $[-2^{i+1},0]$. Hence any embedding of the two paths extends to at most $\Delta^{2^i}$ embeddings of graphs $G_s(i)$ (some of which may be isomorphic). In total, this shows that 
\[n(v,i) \leq \Delta ^{5 \cdot 2^i}\]
for any $v \in V(H)$ and every $i \in \N$.

Define the set $\Sym(v) := \bigcap _{i \in \N} \Sym(v,i)$. We will show that at least one of the sets $\Sym(v)$ must be large enough to derive a contradiction to the above bound. Clearly, if there is an embedding of $G_s$ into $H$, then $s$ is contained in $\Sym(v)$ where $v$ is the image of $0$ under this embedding. Since by assumption every $G_s$ is a subgraph of $H$, this implies that $\Sym = \bigcup _{v \in V(H)} \Sym(v)$. It immediately follows that one of the $\Sym(v)$ must be uncountable. Unfortunately, this is not sufficient for our purpose as we need some control over the sets $\Sym(v,i)$; this is achieved by endowing $\Sym$ with a probability measure.

Note that any $s \in \Sym$ bijectively maps $[2^{i-1}+1,2^i]$ to itself. If we let $s_i$ be the permutation that $s$ induces on $[2^{i-1}+1,2^i]$, then the map $s \mapsto (s_i)_{i \in \mathbb N}$ is a bijection between $\Sym$ and $\prod_{i \in \N} S_{2^i}$ where $S_n$ denotes the symmetric group on $n$ elements. For $\sigma \in S_{2^i}$, we let $\Sym_{\sigma} =\{s \in \Sym \mid s_i = \sigma\}$. It is well known (for instance by the Kolmogorov extension theorem) that there is a probability measure on $\prod_{i \in \N} S_{2^i}$ whose projection on every factor is the uniform probability measure. By the above bijection, this defines a probability measure $\Prob$ on $\Sym$ such that $\Prob(\Sym_{\sigma}) = \frac{1}{(2^i)!}$ for every $\sigma \in S_{2^i}$.
The sets $\Sym(v,i)$ can be written as finite unions of intersections of finitely many $S_\sigma$, hence they are measurable with respect to $\Prob$, and thus the same is true for the sets $\Sym(v)$. By subadditivity of the probability measure we have \[\sum _{v \in V(H)}\Prob(\Sym(v)) \geq \Prob\left(\bigcup _{v \in V(H)} \Sym(v)\right) = \Prob(\Sym) = 1 >0,\] 
and thus there is at least one $v \in V(H)$ such that $\Prob(\Sym(v)) > 0$. From now on fix such a vertex~$v$ and let $\epsilon = \Prob(\Sym(v)) > 0$. Since $\Sym(v) \subseteq \Sym(v,i)$ we have
\[
    \epsilon \leq \Prob(\Sym(v,i)) = \Prob\left(\biguplus _{\sigma \in S_{2^i}} \Sym(v,i) \cap \Sym_\sigma\right) = \sum_{\sigma \in S_{2^i}} \Prob(\Sym(v,i) \cap \Sym_\sigma)
\]
for every $i \in \N$. Each summand on the right hand side is bounded above by $\Prob(\Sym_{\sigma}) = \frac{1}{(2^i)!}$, hence at least $\epsilon (2^i)!$ of the summands must be non-zero.

We claim that $n(v,i)$ is at least half as large as the number of non-zero summands in the above sum. To this end, it suffices to show that $G_s(i)$ and $G_t(i)$ are isomorphic if and only if the restrictions of $s$ and $t$ to $[2^i]$ either coincide or are inverses of one another. Note that the path $-2^{i+1},-2^{i+1}+1, \dots 2^{i+1}$ is the unique spanning path of $G_s(i)$ with no two consecutive vertices of degree $3$. Any isomorphism from $G_s(i)$ to $G_t(i)$ must map this path to its counterpart in $G_t(i)$, and there are precisely two ways of doing this.

We have thus shown that 
\[
\frac \epsilon 2 \cdot (2^i)! \leq n(v,i) \leq \Delta^{5 \cdot 2^i}
\]
for arbitrary $i$.
For large enough $i$ this yields a contradiction since the left hand side asymptotically grows faster than the right hand side.
\end{proof}

\bibliographystyle{abbrv}
\bibliography{bibliography.bib}

\end{document}